\documentclass[12pt, reqno]{amsart} 
\usepackage{amsmath, amsthm, amscd, amsfonts, amssymb, graphicx, color}
\usepackage[bookmarksnumbered, colorlinks, plainpages]{hyperref}
\hypersetup{colorlinks=true,linkcolor=red, anchorcolor=green,
citecolor=cyan, urlcolor=red, filecolor=magenta, pdftoolbar=true}

\catcode`\@=11
\def\omathop#1#2#3{\let\temp=#1\def\letter{#2}
\ifcat#3_ \let\next\@@olim\else\let\next\@olim\fi\next#3}
\def\@olim{\letter\text{-}\!\temp}
\def\@@olim_#1{\mathchoice{
\setbox0=\hbox{$\displaystyle\letter\text{-}\!\temp\!\text{-}\letter$}
\setbox2=\hbox{$\displaystyle\temp$}
\setbox4=\hbox{$\scriptstyle#1$}
\dimen@=\wd4 \advance\dimen@ by -\wd2 \divide\dimen@ by2
\def\next{\letter\text{-}\!\temp_{\hbox to 0pt{\hss$\scriptstyle#1$\hss}}
\hskip\dimen@}
\ifdim\wd2>\wd4 \def\next{\@olim_{#1}}\fi
\ifdim\wd4>\wd0 \def\next{\mathop{\llap{$\letter$-}\!\temp}\limits_{#1}}\fi
\next}
{\@olim_{#1}}{\@olim_{#1}}{\@olim_{#1}}}

\newcommand{\be}{\begin{equation}}
\newcommand{\ee}{\end{equation}}

\def\phi{\varphi}

\newcommand{\bi}{\begin{itemize}}
\newcommand{\ei}{\end{itemize}}
\newcommand{\bn}{\begin{enumerate}}
\newcommand{\en}{\end{enumerate}}

\newcommand{\hide}[1]{} 

\newtheorem{theorem}{Theorem}[section]

\newtheorem{proposition}[theorem]{Proposition}
\newtheorem{corollary}[theorem]{Corollary}
\newtheorem{lemma}[theorem]{Lemma}
\newtheorem{remark}[theorem]{Remark}

\theoremstyle{example}

\theoremstyle{definition}
\newtheorem{definition}[theorem]{Definition}

\numberwithin{equation}{section}

\begin{document}

\title[Frattinian nilpotent Lie algebras ]{Frattinian nilpotent Lie algebras }

\author[M. KianMehr \& F. Saeedi ]
{Mehri KianMehr$^1$ , and Farshid Saeedi$^2$}  



\begin{abstract}
We introduce a novel concept Frattinian nilpotent Lie algebra. Along with some examples, we show that every Frattinian nilpotent Lie algebra has a central
decomposition of its ideals.
\end{abstract}

\maketitle

\section{Introduction}
Deeper studies in nilpotent Lie algebras are associated to special classes. One of these classes is Heisenberg Lie algebras belonging to finite-dimensional nilpotent Lie algebra, and many recent articles are done in this topic. Some generalizations of Heisenberg Lie algebras were studied in \cite{1,2,3,4,5,6,7,8,9,13}.

In the present paper, we introduce a class of finite-dimensional nilpotent Lie algebra, which has similar properties to Heisenberg Lie algebras. In group theory, this concept is Frattinian $ p $-groups,  called by Schmid \cite{10}, which has become the basis of much research in this field. Next section is devoted to introducing Frattinian nilpotent Lie algebras and illustrating some examples.
Then we state and prove some lemmas, which are essential for the main result. In the last section, we state and prove the main result. 
\section{Definitions and preliminary results}
Let $L$ be a finite-dimensional nilpotent Lie algebra. We show the ith term of  lowerand upper central series and Frattinian subalgebra of $L$ by $L^i$, $Z_i(L)$, and $\Phi(L)$, respectively. We recall that the Lie algebra $L$ is Heisenberg if $L^2=Z(L)$ and $dim\ L^2=1$. These algebras are dimension odd with the basis $x,x_1,\ldots,x_{2m}$ associated with the nontrivial multiplication $[x_{2i-1},x_{2i}]=x$ for all $i=1,\ldots, m$. The abelian Lie algebra of dimension $n$ and Heisenberg Lie algebra of dimension $2m + 1$ are denoted by $A(n)$ and $H(m)$, respectively.

Now we define the Frattinian nilpotent Lie algebra.
\begin{definition}
The finite-dimensional nilpotent Lie algebra $L$ is  called Frattinian if for every maximal subalgebra $M$ of $L$, we have $Z(L)\neq Z(M)$.
\end{definition}
It is well known that every finite-dimensional abelian Lie algebra is Frattinian.
Now we state some examples of nonabelian  Frattinian Lie algebras. First we recall the central product of Lie algebras from \cite{11}.
\begin{definition}
The Lie algebra $L$ is a central product of $A$ and $B$ if $L=A+B$ in which $A$ and $B$ are ideals of $L$ such that $[A,B]=0$ and $A\cap B=Z(L)$, which we show it by  $L=A \dotplus B$. 
\end{definition}
If  $L=A \dotplus B$, then $Z(A)\cap Z(B)=A\cap B$   and  $Z(L)=Z(A)+Z(B)$.

The following lemma shows that every Heisenberg Lie algebra is a central product of some of its ideals.

\begin{lemma}\cite[Lemma 3.3]{12}\label{l1}
Let $L$ be a Heisenberg Lie algebra of dimension $2m+1$. Then $L$ is a central product of ideals $B_j$, $1\leq j \leq m$, such that $B_j$ is a Heisenberg Lie algebra of dimension three.
\end{lemma}

In the following proposition, we give an example of nonabelian Frattinian  Lie algebras.
\begin{proposition}
Every Heisenberg Lie algebra is a Frattinian  Lie algebra.
\end{proposition}
\begin{proof}
Let $L$ be a Heisenberg Lie algebra of dimension $2m+1$ and let $M$ be a maximal subalgebra of $L$. Since $dim\ Z(L)=1$, so $dim\ M \cap Z(L)=1$ and $ Z(L) \subseteq M $. Using Lemma \ref{l1}, we have $ L=B_1 \dotplus B_2 \dotplus \cdots \dotplus B_m $, where  $ B_j\simeq H(1) $ for $ 1 \leqslant j  \leqslant m $.

Since $M$ is a maximal subalgebra of $L$, there exists $ 1 \leqslant k  \leqslant m $ such that $ B_k \nleqslant M $, which implies that  $ L=B_k + M $. Moreover,  since $ dim\ M= 2m $, so $ dim\ (M \cap B_k )=2 $. Thus  $ M\cap B_k $ is an abelian Lie algebra contained $ Z(M) $. Therefore, $Z(L)\neq Z(M)$, since $dim\ Z(L)=1$.
\end{proof}
 Now we give an example of nilpotent Lie algebra that is not a Frattinian  Lie algebra. First, we recall that  $L$ Lie of algebra of dimension  $n$ is said to be filiform if $ dim\ \frac{L}{L^2}=2 $ and  $ dim\ \frac{L^i}{L^{i+1}}=1 $ for all $ i=2,3, \ldots , n-1 $. In fact, every $n$-dimensional filiform Lie algebra $L$ is a Nilpotent Lie algebra with nilpotency class $n-1$ and   $Z_i (L)=L^{n-i}$ for all $ i \geqslant 1 $. We know that $H(1)$ is a Heisenberg and filiform Lie algebra, as well as a Frattinian  Lie algebra.

 \begin{proposition}
 Let $L$ be a filiform Lie algebra of dimension greater than or equal  to four. Then $L$ is not a Frattinian  Lie algebra. 
 \end{proposition}
 \begin{proof}
 Since $L$ is a nonabelian Lie algebra of dimension greater than or equal  to four, it has more than two maximal subalgebra. 

Moreover, $L$ at most has one abelian maximal subalgebra. For this, let $A$ and $B$ be two abelian maximal subalgebras of $L$. Then $L$ is the central product of $A$ and $B$ and 
\[\frac{L}{Z(L)}=\frac {L}{A\cap B}\simeq \frac{L}{A}\oplus\frac{L}{B}.\]
Since $dim\ Z( L)=1$, so $dim \ L=3$, which is a contradiction. 
  In other words, $L$ is a nonabelian Lie algebra that has at least two different maximal nonabelian subalgebras  $M$ and $N$. Let  $ dim\ \frac{L}{Z(M)} = m$, $ dim\ \frac{L}{Z(N)}=n $, and $dim\ L =l \geqslant 4$, where   $ m , n \in \lbrace 2 , 3 , \ldots , l-1 \rbrace $.
 
 Since  $ \dfrac{L}{Z(M)} $ is a nilpotent Lie algebra with nilpotency class at most $m-1$, we have  $ L^m \subseteq Z(M) $. The fact that $L$ is a filiform Lie algebra implies that
\[m = dim\ \frac{L}{L^m}= dim\ \frac{L}{Z(M)} + dim\ \frac{Z(M)}{L^m} = m + dim\ \frac{Z(M)}{L^m}. \]

Thus  $ Z(M)=L^m $. Similarly one can show that $Z(N)=L^n$. Without loss of generality we assume that $ m \leqslant n $. Since $ Z(M) \cap Z(N) =Z(L) $, so 
\[ Z(L) =Z(M) \cap Z(N) =L^m \cap L^n =L^n =Z(N), \]
which implies that $ L $ is not a Frattinian  Lie algebra.
 \end{proof}
 In the rest of this section, we state and prove some lemmas, which are essential in the proof of main theorem.

\begin{lemma}\label{l2}
Let $M$ be a maximal subalgebra of finite-dimensional nilpotent Lie algebra $L$.
Then  $  C_L(M)=Z(M)$  or $  C_L(M)=Z(L)$ is a supplement to $M$ in $L$, where  $C_L(M)$ is a  centralizer of  $M$  in  $L$.
\end{lemma}
\begin{proof}
We consider two cases:
1) Let $M$ be an abelian subalgebra of $L$. If  $C_L(M)=M$, then $C_L(M)=Z(M)=M$. Now assume that $C_L(M)=L$.  If $Z(L)=L$, then $C_L(M)=L=Z(L)$ and $L=C_L(M)+M$, as a desired.
Now we show that $Z(L)=M$ does not hold. If  $Z(L)=M$, then $dim\ \frac {L}{M}=dim\ \frac{L}{Z(L)}=1$, which contradicts with the fact that $L$ is a nonabelian Lie algebra.

2) Let $M$ be a  nonabelian subalgebra of $L$. 
If $  C_L(M)\neq Z(M)$, then  $C_L(M)\nleq M$ . Since $L=C_L(M)+M$, so $L$ is a central product of  $ M $ and $C_L(M)$. Therefore,
\[Z(L)=Z(C_L(M)) +Z(M) \]
and
\[\frac{L}{M}=\frac{C_L(M)+M}{M}\simeq \frac{C_L(M)}{M \cap C_L(M)}= \frac{C_L(M)}{Z(M)}. \]

On the other hand, $  Z(M)$ is an ideal of $ Z(C_L(M)) $ . Hence $C_L(M)$ is abelian. Therefore,  $Z(C_L(M))=C_L(M)$  and 
\[ Z(L)=Z(C_L(M))+Z(M)=C_L(M). \]
\end{proof} 
In the proof of Lemma \ref{l2}, we obtain that  $C_L(M)\nleq M$, which implies that $Z(L) \nleq M$. 
\begin{corollary} \label{7}
Let $L$ be a finite-dimensional nilpotent Lie algebra and let $M$ be a maximal subalgebra of $L$. If $Z(L) \nleq M$, then $C_L(M) =Z(L)$, and if $Z(L) \leqslant M$, then $C_L(M) =Z(M)$.
\end{corollary}

\begin{proof}
If $ Z(L) \nleqslant M $, then similar to the proof of Lemma \ref{l2}, we have $ C_L(M) = Z(L) $. Now assume that $ Z(L) \leqslant M $. Since $ M \cap C_L(M) =Z(M) $, it is sufficient to show that $ C_L(M) \leqslant M $. We assume on the contrary  that there exists $ x \in C_L(M) \setminus M $.  Thus $ L= M + \langle x \rangle  $, and for  each element $ l = kx+m $ of  $ L $, where $ k \in \mathbb{Z} $ and $ m \in M $, we have
\[  [x , l] =[x,kx+m]=0 .\]
that is,   $ x \in Z(L) \leqslant M $, which is a contradiction.
\end{proof}

\begin{lemma} \label{8}Let $L$  be a finite-dimensional nilpotent Lie algebra and let $M$ be a maximal subalgebra of $L$. If $ Z(M) \nleqslant Z(L) $, then $ C_L(Z(M))=M $. 

\end{lemma}
\begin{proof}
It is straightforward.
\end{proof}
\begin{lemma} \label{9}
Let $L$ be a Frattinian nonabelian nilpotent Lie algebra of dimension finite such that every maximal subalgebra $M$  of $L$ containing $ Z(L) $ satisfies $ Z(M)\leqslant Z(L) + L^2 $. Then the following properties hold:
\\
1) $ M= C_L (Z(M)) $,\\
2) $ C_L(Z(L^2)) = Z(L) + L^2 $,\\
3) If $ F $ is the smallest supplement subalgebra of   $ Z(L) $ in $L  $, then $ F $ is Frattinian, 
 $ Z(F) \leqslant F^2 $, $ C_L(F)=Z(L) $, and  $ C_L(Z(F^2))=F^2 $.
\end{lemma}
\begin{proof}
1)It is well known that $ Z(L) \lvertneqq Z(M) $. Since $ Z(M) \nleqslant Z(L) $, Lemma \ref{8} implies that $M=C_L(Z(M))  $.
\\
2) Let $ J $ be the subalgebra generated by the union of center of these maximal subalgeras of $ L $. Then $ J \subseteq Z(L) + Z(L^2) $ and
\begin{align*}
C_L(J) \leqslant \cap _{Z(M) \subseteq J} C_L(Z(M))&= Z(L) + L^2 \\
&=C_L(Z(L^2)).
\end{align*}
\\
3) It is clear that $ F $ is an ideal of $ L $. Since $ F $ is the smallest supplement subalgebra of $ Z(L) $ in $L$, from \cite{14} we have $ F $ is a nilpotent Lie algebra with $ \Phi (F) =F^2 $. Therefore $ Z(F) \subseteq F^2 $ and since $ L=F+ Z(L) $, we have $ C_L(F)=Z(L) $.

Suppose that $X$ is a maximal subalgebra $ F $. Then $ M :=Z(L) + X $ is a maximal subalgebra of $ L $, since $ Z(F)= X \cap Z(L) $. The properties of $ F $ imply that $L \neq Z(L) + X $ and thus
\[ dim\ \frac{L}{M} = dim\ \frac{F+Z(L)}{X+Z(L)} = dim\ \frac{F}{Z(F)} - dim\ \frac{X}{X \cap Z(L) } = dim\ \frac{F}{X}  =1,\]
that is, $M$ is a maximal subalgebra of $L$.

Now we show that $F$ is a  Frattinian subalgebra of $L$. If it is not, then  $ Z(X)=Z(F) $ and $ M $ the central product of  $Z(L) $ and $ X $. Therefore,
\[ Z(M)=Z(L)+Z(X) = Z(L)+Z(F) = Z(L), \]
which contradicts with $L$ is  Frattinian.

We clear that $ F^2=\Phi (F) \leqslant C_F(Z(F^2)) $ and that $ L^2 =[Z(L) +F , Z(L)+F]=F^2 $. If $ C_F(Z(F^2)) \nleqslant F^2 =\Phi(F) $, then there exists a maximal subalgebra  $N  $ of $ F $ such that  $C_F(Z(F^2)) \nleqslant N  $. Hence, there exists an element $ a \in C_F(Z(F^2)) \setminus N $ such that $ F = \langle a \rangle + N $, but $ Z(L) + N $ is a maximal subalgebra of $ L $, so $ a \notin Z(L) + N$.

The above discussion implies that
\begin{align*}
a \in C_F(Z(F^2)) \leqslant C_L(Z(F^2)) &=Z(L) + L^2 \\
&=Z(L) +F^2 \\
&\leqslant Z(L) +N,
\end{align*}
which is a contradiction.
\end{proof}

\begin{lemma} \label{10}
Let $L$ be a Frattinian nonabelian nilpotent Lie algebra of dimension finite such that there exists a maximal subalgebra $M$  of $L$ containing $ Z(L) $ such that $ Z(M)\nleqslant Z(L) + L^2 $. Then the following properties hold:
\\
1) There exists a maximal subalgebra $ N $ of  $  L$ including $  Z(L)$ such that $ Z(M) \nleq N $.
\\
2) $ Z(N) \cap M = Z(M) \cap N =Z(L) $. 

\end{lemma}

\begin{proof}
Since $L$ is a  Frattinian  Lie algebra and  $ Z(L) \lneqq M $, we have $ Z(L) +L^2 \lneqq Z(M) + L^2 \lneqq L $. Let $ \frac{T}{Z(L) +L^2} $ be a supplement of $ \frac{Z(M) + L^2}{Z(L) +L^2} $ in $ \frac{L}{Z(L)+L^2} $ and let  $\frac{N}{Z(L)+L^2}  $ be a maximal subalgebra of $\frac{L}{Z(L) + L^2}  $ including $ \frac{T}{Z(L)+ L^2} $ and not  $ \frac{Z(M) +L^2}{Z(L)+L^2} $. Then

\begin{align*}
Z(L) &\leqslant Z(L) + L^2 \\
&\nleqslant Z(M) +L^2 \\
&\nleqslant N
\end{align*}
and $ dim\ \frac{L}{N}=1 $.

Hence $ N $ is a maximal subalgebra of $L  $ containing $ Z(L) + L^2 $ and not $ Z(M) + L^2 $. Since  $L^2 \leqslant N  $, we have $ Z(M) \nleqslant N $.

For the second case, we know that $ N $ is a maximal subalgebra of  $ L $ containing $ Z(L) $, which implies that $ Z(L) \lneqq Z(N) $. Since  $ L $ is Frattinian, we have $ Z(L)\nleqslant Z(N) $. Lemma  \ref{8}  implies that $ C_L(Z(N))=N $ and similarly $ C_L(Z(M))= M $.

Now we prove that $ Z(M) \cap N = Z(L) $. If it is not, then $ C_L(Z(M) \cap N)=M $. Corollary \ref{7} implies that $ C_L(N)=N $. Therefore
\[ M= C_L(Z(M)\cap N) \geqslant C_L(N) =Z(L). \]
Thus $ N= C_L(Z(N)) \geqslant C_L(M) =Z(M) $,  a contradiction.
Similarly one can prove that  $ Z(N) \cap M = Z(L) $.
\end{proof}
\section{Main result}
In this section, using the lemmas of section 2, we state and prove the main result, which shows that every Frattinian nilpotent Lie algebra has a central  decomposition of its ideals.
\begin{theorem}
Let $L$ be a Frattinian nonabelian nilpotent Lie algebra of dimension finite. Then one of the following conditions holds:

1) $L$ is a central product of nonabelian nilpotent Lie algebras of dimension $ 2+ dim\ Z(L) $ that their center is equal to the center of $ L $.

2)$ L=E \dotplus F $ is a central product   Frattinian  Lie algebras   $ E $ and  $ F $ with $ C_L(Z(F^2))=F^2 $ and $ E=C_L(F) $ such that $ E^2 \subseteq Z(L) $.

Moreover, in the second case,  $ E=Z(L) $ or $ E $ has a central product as the same as the first case.

\end{theorem}
\begin{proof}
If for every maximal subalgebra $ M $ of  $ L $ containing $ Z(L) $ we have  $ Z(M) \nleqslant Z(L) + L^2 $, then Lemma \ref{9} completes the proof.
 
 We assume that there exists a maximal subalgebra $ M $ of  $ L $ containing  $ Z(L) $ such that $ Z(M) \nleqslant Z(L) +L^2 $. Lemma \ref{10} implies that there exists a maximal subalgebra  $ N $ of $ L $  containing  $ Z(L) $  such that
\[
Z(M) \cap N=Z(N)\cap M=Z(L).
\]Therefore,
\[
Z(M) \cap Z(N) = Z(L).
\]
Now we set $ E_1 :=Z(M)+Z(N) $. Hence $ E_1 $ is a nonabelian ideal of  $ L $ with dimension $ dim\ Z(L)+2 $ that satisfies  $ Z(E_1)=Z(L) $. If we set $ L_1:=C_L(E_1) $, then $ L_1 = M\cap N $ and  $ L $ is the central product of $ E_1 $ and  $ L_1 $ with $ Z(L_1)=Z(L) $.

We claim that $ L_1 $ is a Frattinian  Lie algebra. Let  $ M_1 $ be a maximal subalgebra of $ L_1 $. If $ M_1 $ does not contain $ Z(L) $, then $ Z(M_1)\neq Z(L_1) = Z(L) $, as a desired. Suppose that $ M_1 $ contains $ Z(L) $. Then $ E_1 + M_1 $ is a maximal subalgebra of $ L $ such that
\[ Z(M_1) = C_{L_1}(M_1)= C_L(E_1+ M_1) = Z(E_1 +M_1). \]
Since $ L $  is Frattinian, $ Z(L) \neq Z(E_1 + M_1) $, that is, $ Z(M_1) \neq Z(L_1) = Z(L) $ and $ L_1 $ is  Frattinian.

Continuing this process, we assume that  $  L_1$ has a maximal subalgebra, say   $ M_1 $, that contains $ Z(L) $ and $ Z(M_1) \nleq Z(L_1)+L^2_1 $. Therefore, there exist ideals  $ L_2 $ and  $ E_2 $ such that $ L_1 $ is the central product of them and 
\[
Z(L_2)=Z(E_2)=Z(L_1)=Z(E_1)=Z(L).
\]
Moreover,  $ L_2=C_{L_1}(E_2)=C_L(E_1+E_2) $, $ dim\ E_2 = 2+dim\ Z(L) $, and $ E_2$ is nonabelian. Thus $ L=E_1+E_2+L_2 $.

This process will be stop after $n$ steps, so we have $ L=E_1+E_2+ \cdots + E_n+L_n $, where $ L_n=C_{L_{n-1}}({E_n})=C_L(E_1+ \cdots +E_n) $,   $ E_i $ are nonabelian ideals  of dimension $ 2+dim\ Z(L) $ and $ L_n $ does not have a maximal subalgebra such as $ M_n $  containing  $ Z(L) $ such that $ Z(M_n) \nleq Z(L_n)+L^2 $.

In other words, for every maximal subalgebra  of $  L_n$, one of the following cases holds: 1)  there is no maximal subalgebra of $  L_n$ containing $ Z(L_n)=Z(L) $, or 2) for every maximal subalgebra $  M_n$ of $ L_n $  containing $ Z(L_n)=Z(L) $, we have  $ Z(M_1) \leqslant Z(L_n) +L^2_n $. In the first case, we have $ L_n =Z(L_n)=Z(L) $. Since  $ Z(L) \leqslant E_i $, so \[
L=E_1 + E_2 + \cdots + E_n+ Z(L) = E_1 + \cdots + E_n,
\]
as a desired.

For the second case, Lemma \ref{9} implies that if $ F $ is the smallest supplement subalgebra of $ Z(L_n)=Z(L) $ in $ L $, then  $ Z(F)\subseteq F^2 $, $ C_{L_n}(F) =Z(L_n) $, and $ C_L(Z(F^2)) = F^2 $.
Now set $ E := E_1 + E_2 +\cdots + E_n $. Then $ C_L(F) = E $ and   $ L $ is the central product of $ E $ and $ F $.

We know that $ \frac{E}{Z(L)} $ is an abelian Lie algebra of dimension $ 2n $. Now we show that  $ F $ is Frattinian.  We assume on the contrary that $F$ is not Frattinian. If $ X $ is a maximal subalgebra of $ F $ with $ Z(X)=Z(F)$, then $ X+Z(L) $ is a maximal subalgebra of  $ L_n $  and
 \begin{align*}
 Z(X+Z(L)) &=Z(F)+Z(L) \\
 &= Z(L) \\
 &= Z(L_n),
\end{align*}  
which contradicts with the fact that  $ L_n $ is Frattinian.

Similarly  $ E $ is Frattinian.  We assume on the contrary that $E$ is not Frattinian. If $ X $ is a maximal subalgebra of $ E $ with $ Z(X)=Z(E)=Z(L) $, then $ X+F $ is a maximal subalgebra of  $ L $  and
\begin{align*}
Z(X+F)=Z(X)+Z(F) &=Z(L)+Z(F) \\
&=Z(L),
\end{align*}
which contradicts with the fact that  $ L $ is Frattinian.
\end{proof}

 \begin{remark}
 We note that the central decomposition of  $ L=E+F $ is not unique in general, but since $ Z(L)+F^2=Z(L)+L^2 $, the centralizer $ C_L(F^2)=C_L(L^2) $ is independent of the choice of $ F $. Moreover, $ C_L(F^2) = E+Z(F^2) \supseteq E+L^2 $.
 
 \end{remark}

\vskip 0.9 true cm

{\tiny $^{1,2}$Department of Mathematics, Mashhad Branch, Islamic Azad University, Mashhad, Iran.
\vskip 0.2 true cm		
	{\it E-mail address:} me.kianmehr@yahoo.com 
\vskip 0.1 true cm	
	{\it E-mail address:} saeedi@mshdiau.ac.ir

\end{document}